\documentclass[12pt]{article}%
\usepackage[utf8]{inputenc}
\usepackage[all]{xy}

\usepackage{amsmath}
\usepackage{amsfonts}
\usepackage{amssymb}
\usepackage{graphicx}
\providecommand{\U}[1]{\protect\rule{.1in}{.1in}}

\newcommand{\N}{{\mathbb N}}
\newcommand{\C}{{\mathbb C}}

\newcommand{\R}{{\mathbb R}}
\newcommand{\Z}{{\mathbb Z}}

\newcommand{\T}{{\mathbb T}}

\newcommand{\Ho}{{\mathbb H}}

\newcommand{\Hh}{{\mathcal H}}
\newcommand{\Aa}{{\mathcal A}}
\newcommand{\Cc}{{\mathcal C}}

\newcommand{\Ss}{{\mathcal S}}

\def\ov{\overline}

\def\wh{\widehat}
\def\bi{\begin{itemize}}
\def\ei{\end{itemize}}
\def\un{\underline}

\newcommand{\ba}{\begin{eqnarray}}
\newcommand{\ea}{\end{eqnarray}}
\newcommand{\bas}{\begin{eqnarray*}}
\newcommand{\eas}{\end{eqnarray*}}
\newcommand{\be}{\begin{equation}}
\newcommand{\ee}{\end{equation}}

\def\dxj{\partial_{x_j}}
\def\dx0{\partial_{x_0}}

\newtheorem{theorem}{Theorem}
\newtheorem{proposition}[theorem]{Proposition}
\newtheorem{corollary}[theorem]{Corollary}
\newtheorem{lemma}[theorem]{Lemma}

\newenvironment{proof}[1][Proof]{\noindent\textbf{#1.} }{\ \rule{0.5em}{0.5em}}
\newtheorem{preremark}[theorem]{Remark}
\newenvironment{remark}{\begin{preremark}\rm}{\hfill$\Diamond$\end{preremark}}
\newtheorem{prenotation}[theorem]{Notation}

\numberwithin{equation}{section}
\numberwithin{theorem}{section}
\begin{document}

\title{{Coherent State Transforms and the \break Weyl Equation in Clifford Analysis}}
\author{Jos\'e  Mour\~ao\thanks{Department of Mathematics and Center for Mathematical Analysis, Geometry and Dynamical Systems, Instituto Superior T\'ecnico, University of  Lisbon.}, \, Jo\~ao P. Nunes\footnotemark[1] \, and Tao Qian\thanks{Department of Mathematics, Faculty of Science and Technology, University of  Macau.}}

\maketitle

\bigskip

\begin{abstract}

We study a transform, inspired by coherent state transforms, from the Hilbert space of Clifford algebra 
valued square 
integrable functions 
$L^2(\R^m,dx)\otimes \C_{m}$ to a Hilbert space of solutions of the Weyl equation on $\R^{m+1}= \R\times \R^m$, namely to the Hilbert space ${\mathcal M}L^2(\R^{m+1},d\mu)$  of $\C_m$-valued monogenic functions on 
$\R^{m+1}$ which are 
$L^2$ with respect to an appropriate measure $d\mu$. We prove that this transform is a unitary isomorphism of Hilbert spaces and that it is therefore an analog of the Segal-Bargmann transform for Clifford analysis. As a corollary 
we obtain an orthonormal basis of monogenic functions on $\R^{m+1}$. 
We also study the case when $\R^m$ is replaced by the $m$-torus $\T^m.$  
Quantum mechanically, this extension establishes the unitary equivalence of the Schr\"odinger 
representation on $M$, for $M=\R^m$ and $M=\T^m$, with a representation on the Hilbert space ${\mathcal M}L^2(\R\times M,d\mu)$ of
solutions of the Weyl equation on the space-time $\R\times M$.

\vskip 0.2cm


\end{abstract}

\tableofcontents

\section{Introduction}

In this work we continue to explore the extensions of coherent state transforms to the context of Clifford 
analysis initiated in \cite{KMNQ}.

Clifford analysis (see \cite{BDS, DSS}) extends the theory of complex analysis of holomorphic
functions to functions of Clifford algebra variables, obeying  generalized Cauchy-Riemann conditions and called 
monogenic functions. In the context of the present paper, monogenic functions correspond to solutions 
of the Weyl equation in the euclidean space-time $\R\times M$, for $M=\R^m$ or 
$M=\T^n$.
In quantum physics, Clifford algebra or spinor representation valued
functions describe some systems with internal degrees of freedom, such as particles with spin. Notice that
spinor valued solutions of the Dirac equation can be described by Clifford algebra valued solutions of the 
same equation, by decomposing the algebra in a sum of minimal left ideals. 
(See, for example, Chapter 2 of \cite{DSS}.)

On the other hand, the Segal-Bargmann transform \cite{Ba, Se1, Se2}, for a particle on $\R^m$, establishes
the unitary equivalence of the Schr\"odinger representation with Hilbert space $L^2(\R^m, dx)$, with
(Fock space-like) representations with Hilbert spaces, ${\mathcal H}L^2(\C^m, d\nu)$, of holomorphic functions
on the phase space, $\R^{2m}\cong \C^m$ which are $L^2$ with respect to an appropriate measure $d\nu$. 
In the Schr\"odinger representation, the position operators $\hat x_j, j=1,\dots, m,$ act diagonally while the momentum
operator $\hat p_j= -i \frac{\partial}{\partial x_j}$ . In the Segal-Bargmann representation, on the other hand, 
it is the operators $\widehat{x_j+ ip}_j$ that act on the Hilbert space ${\mathcal H}L^2(\C^m, d\nu)$ 
as multiplication by the holomorphic functions $x_j + ip_j, j=1,\dots, m$.
In \cite{Ha1}, Hall has defined coherent state transforms (CSTs) for compact Lie groups $G$
which are analogs of the Segal-Bargmann transform. 

Let $\R_m$ (respectively $\C_m$) be the real (complex) Clifford algebra with $m$ generators, see section \ref{ss-ca}).
In \cite{KMNQ}, we presented a generalization of the Segal-Bargmann transform to a transform taking functions in 
$L^2(\R, dx)$ to Hilbert spaces of slice or axial monogenic Clifford algebra valued functions on $\R^{m+1}$. The unitarity 
of these transforms, with respect to appropriate measures, was established.  See also \cite{DG}, where a 
similar transform (for $m=2$) was studied, but with range a Hilbert space 
of slice monogenic functions on the full quaternionic algebra $\R_2={\mathbb H}$. 

In the present work, we give a different generalization of the Segal-Bargmann transform. Instead of going from 
functions on $\R$ to functions on $\R^{m+1}\subset \R_m$, as in \cite{KMNQ}, this transform adds a single time variable 
to functions on $\R^m$ and maps to solutions of the Weyl (or Cauchy-Riemann) equation on $\R^{m+1}$. Namely, the transforms studied in sections \ref{s-3} and \ref{s-4} give a unitary 
equivalence between Schr\"odinger quantization on $L^2(M,d\un x)\otimes \C_{m}$
and a Hilbert space of solutions of the Weyl equation in the euclidean space-time $\R\times M$.

\section{Preliminaries}
\label{s-2}

\subsection{Coherent state transforms (CST)}
\label{cst}

In  \cite{Ha1}, Hall  introduced  a class of unitary integral transforms 
from the Hilbert spaces of square integrable functions on compact 
Lie groups $G$, with respect to the Haar measure, to spaces of holomorphic functions on the complexification 
$G_\C$, which are $L^2$ with respect to an appropriate measure.
 These are known as coherent state transforms (CSTs) 
or generalized Segal--Bargmann transforms. 
These transforms were extended to Lie  groups of compact type, which include the case of
$G=\R^m$ considered in the present paper, by Driver in \cite{Dr}. General Lie
groups of compact type are products of compact Lie groups and $\R^m$, see Corollary 2.2 of \cite{Dr}.

We will briefly recall now
the case $G=\R^m$
for which the Hall transform coincides with the classical
Segal--Bargmann transform \cite{Ba, Se1, Se2}.
 
Let $\rho_t(x)$  denote the fundamental solution of the
heat equation.
$$
\frac{\partial}{\partial t} \rho_t =  \frac 12 \Delta \, \rho_t ,
$$
i.e.
$$
\rho_t(x) = \frac{1}{(2\pi t)^{m/2}} \, e^{-\frac{|x|^2}{2t}},
$$
where $\Delta$ is the Laplacian for the euclidean metric.
The Segal--Bargman or coherent state transform
$$
U \ : L^2(\R^m, dx) \longrightarrow \Hh(\C^m)
$$
is defined by
\ba
\label{ee-cst} \nonumber
U(f) (z) &=& \int_{\R^m} \, \rho_{t=1}(z - x) f(x) \, dx = \\
&=& \frac{1}{(2\pi )^{m/2}} \, \int_{\R^m}  \, e^{-\frac{|z-x|^2}{2}} \, f(x)
\, dx \, .
\ea
where $\rho_1$ has been analytically continued to $\C^m$.
We see that the transform $U$ in (\ref{ee-cst}) factorizes according
to the following diagram

\vskip 0.3cm

\begin{align}
 \label{d1}
\begin{gathered}
\xymatrix{
&&{\mathcal H} (\C^m)  \\
L^2(\R^m, dx)  \ar@{^{(}->}[rr]_{e^{\frac{\Delta}2}}  \ar@{^{(}->}[rru]^{U} && \Aa (\R^m)
\ar[u]_{\Cc}
  }
\end{gathered}
\end{align}

\noindent where $\Cc$ denotes the  analytic continuation from $\R^m$ to $\C^m$
and $e^{\frac{\Delta}2}(f)$ is the (real analytic) heat kernel evolution of the function $f\in L^2(\R^m,dx)$ at time $t=1$,
that is the solution of
\be
\label{ee-he}
\left\{
\begin{array}{lll}
\frac{\partial}{\partial t} \, h_t &=& \frac 12 \Delta \, h_t\\
h_0 &=& f
\end{array}   ,
\right.
\ee
evaluated at time $t=1$,
$$
e^{\frac{\Delta}2}(f) = h_1.
$$
$\Aa(\R^m)$ in (\ref{d1}) is the space of (complex valued) 
real analytic functions on $\R^m$ with unique analytic continuation to entire functions 
on $\C^m$.
Let $\widetilde \Aa (\R^m)\subset \Aa(\R^m)$  denote the image of $L^2(\R^m,dx)$  by the operator
$e^{\frac{\Delta}2}$. The analytic continuation ${\mathcal C}$ on 
$\widetilde \Aa (\R^m)$
can be writen in the 
form
\be
\label{ee-aco}
{\mathcal C} (f) (x, y) = f(x + i y) =
e^{i \sum_{j=1}^m \, y_j \, \partial_{x_j}} \, (f(x))    \, .
\ee
Then, the Segal--Bargmann theorem reads as follows.

\begin{theorem} 
\label{th-cst}
The transform $U$ in the diagram
\begin{align}
 \label{d2}
\begin{gathered}
\xymatrix{
&&  \Hh L^2(\C^m, \nu \,  dx dy)   \\
L^2(\R^m, dx)  \ar[rr]_{e^{\frac{\Delta}2}}  \ar[rru]^{U} &&
\widetilde \Aa (\R^m)
\ar[u]_{\, \, e^{i \sum_{j=1}^m \, y_j \, \partial_{x_j}}}
  }
\end{gathered}
\end{align}
 is a unitary isomorphism, where
$\nu(y) = e^{-|y|^2}$.
\end{theorem}

\subsection{Clifford analysis}
\label{ss-ca}

 Let us briefly recall from \cite{BDS, DSS, CSS1, CSS2, CSS3, DS, LMQ, Q1, Q2, Sou},  some definitions and results from  Clifford analysis.
 Let $\R_{m+1}$ denote the real Clifford algebra with $(m+1)$ generators, $\tilde e_j, j = 0, \dots, m$, identified
with the canonical basis of $\R^{m+1} \subset \R_{m+1}$ and satisfying the relations 
$\tilde e_i \tilde e_j + \tilde e_j \tilde e_i = - 2 \delta_{ij}$. Let $\C_{m+1}=\R_{m+1}\otimes \C$.
We have that $\R_{m+1} = \oplus_{k=0}^{m+1} \R_{m+1}^k$,
where $ \R_{m+1}^k$ denotes the space of $k$-vectors, defined by $\R_{m+1}^0 = \R$ and  $\R_{m+1}^k 
= {\rm span}_\R \{\tilde e_A \, : \, A \subset \{0, \dots , m\}, |A| = k\}$.  

Notice also that $\R_1 \cong \C$ and $\R_2 \cong \Ho$.
The inner product in $\R_{m+1}$ is defined by
$$
(u, v) = \left(\sum_A u_A \tilde e_A ,  \sum_B v_B \tilde e_B\right) = \sum_A u_A v_A .
$$
The Dirac operator is defined as
$$
\tilde D = \sum_{j=0}^m \, \dxj \tilde e_j.
$$
Let 
$$
e_j = -\tilde e_0 \tilde e_j, \, j=1, \dots, m.
$$
Note that $e_ie_j+e_je_i=-2\delta_{ij}$, so that $\{e_j\}_{j=1, \dots ,m}$ is a 
set of generators for a subalgebra $\C_{m+1}^+\subset \C_{m+1}$ with $\C_{m+1}^+ \cong \C_m$. We will henceforth 
consider $\C_m$-valued functions.

One defines the Cauchy-Riemann operator by
\be
\nonumber
D  = \dx0 + \un D ,
\ee
where
$$
\un D = \sum_{j=1}^m \, \dxj  e_j.
$$
We have that $\un D^2 = - \Delta_m$ and
$D \ov D  = \Delta_{m+1}$.

Consider the subspace of $\R_{m}$ of $1$-vectors  
$$
\{\un x = \sum_{j=1}^{m} x_j e_j: \,\, x=(x_1,\dots,x_m)\in \R^m\}\cong \R^m,
$$
which we identify with $\R^{m}$. Note that $\un x^2 = - |\un x|^2 = - (x, x).$

Recall that a continuously  differentiable function $f$ on an open domain ${\cal O} \subset \R^{m+1}$,  
with values on $\C_{m}$, 
is called (left) monogenic on ${\cal O}$  if it satisfies the Weyl, or Cauchy-Riemann, equation (see, for example, \cite{BDS,DSS,Sou})
\be\nonumber
D f(x_0, \un x)  = (\dx0 + \un D) f(x_0, \un x) = 0.
\ee
For $m=1$, monogenic functions on $\R^2$ correspond to holomorphic
functions of the complex variable $x_0+e_1 x_1$.


In order to describe monogenic functions 
let us, following \cite{LMQ, Q2}, introduce the following projectors
$$\chi_\pm(\un p) = \frac12 \left(1\pm \frac{i\un p}{|\un p|}\right),$$ 
satisfying
$$
Q (i \un p) \, \chi_\pm(\un p) = \chi_\pm(\un p) \, Q (i \un p)
= Q (\pm |\un p|) \, \chi_\pm(\un p)
$$
for any polynomial in one variable with complex coefficients, 
$$
Q(\lambda) = \sum_{k=0}^\ell \, c_k \lambda^k.
$$
For any function $B$, of one real variable, one naturally defines the following Clifford 
algebra valued function on $\R^m$,
$$
b(p_1,\dots,p_m) := B(|\un p|)\chi_+(\un p) + B(-|\un p|)\chi_-(\un p). 
$$
By abuse of notation, using the analogy for the case when $B$ is a polynomial, we will 
still denote the right-hand side by $B(i\un p)$. Then, for  $B_{x_0}(y)= e^{-x_0y}$, the Clifford algebra 
valued function
\be\label{be}
e^{-ix_o\un p}=B_{x_0}(i\un p)=b_{x_0}(p_1,\dots ,p_m) = e^{-x_0|\un p|}\chi_+(\un p) + e^{x_0|\un p|}\chi_-(\un p),
\ee
satisfies the equation
$$
\frac{\partial}{\partial x_0} e^{-ix_o\un p} = -i\un p\, e^{-ix_o\un p}, 
$$
which implies that the functions
$$
e(x, \un p) = e^{i ((\un x,  \un p) - x_0 \un p)}
$$
are monogenic in $(x_0,\un x)$ for all $\un p \in \R^m$. In fact, this is the Cauchy-Kowalevski extension
$$
e(x,\un p)=e^{-x_0 \un D} \,e^{i(\un p, \un x)} = e^{i(\un p, \un x)} \,e^{i x_0\un p},
$$
which follows straightforwardly from 
$$
\un D e^{i(\un p, \un x)} = i e^{i(\un p, \un x)} \un p.
$$

These functions play a very important role
in the analysis of monogenic functions. We then obtain the following corollary of \cite{LMQ,Q2}.
Let $\Ss(\R^m)$ be the space of Schwarz functions on $\R^m$. 

\begin{proposition}\label{pu2}
Let $b_{x_0}$ be as in (\ref{be}) and let $f\in \Ss(\R^m)\otimes \C_{m}$ be such that for all $x_0\in \R$ 
\be\label{condition}
b_{x_0}\,\hat f \in \Ss(\R^m)\otimes \C_{m},
\ee
where $\hat f$ denotes the Fourier transform of $f$.
Then, 
\be\label{ft}
F(x_0,\un x)=\frac 1{(2\pi)^{m/2}} \, \int_{\R^m} \, e^{i ((\un p, \un x)-x_0\un p)} \, \wh f(\un p) \, d\un p,
\ee
defines a monogenic function satisfying the Cauchy problem 
\ba\label{supermonogenic}
&& \left\{
\begin{array}{rcl}
\frac{\partial F}{\partial x_0} &=& -  \un D \,F  \\
F(0, \un x) & = & f(\un x) \, . 
\end{array}  
  \right. 
\ea
\end{proposition}

\begin{proof}
We have 
$$
f(\un x)= \frac 1{(2\pi)^{m/2}} \, \int_{\R^m} \, e^{i (\un p, \un x)} \, \wh f(\un p) \, d\un p=F(0,\un x).
$$ 
From above, $e^{i ((\un p, \un x)-x_0\un p)}$ are monogenic functions, for all $p\in \R^m$, and under the conditions of the proposition, we can differentiate under the integral sign in (\ref{ft}) so that $F$ is also a monogenic function.
\end{proof}

Following \cite{DS, DSS} (see \cite{DSS}, chapter III, Section 2 and Theorem 6 in \cite{Som}),  one also has

\begin{proposition}\label{initial}
Let $F$ be a monogenic function on $\R^{m+1}$, with $F(x_0=0,\un x)=f(\un x)$. Then, 
\be\label{ee-cke}
F(x_0,\un x) = \left(e^{-x_0\un D} f\right)(x_0,\un x):=
\left(\sum_{k=0}^\infty \, (-1)^k \frac{x_0^k}{k!} \un D^k f\right)(\un x),
\ee
where the series converges uniformly on compact subsets.
\end{proposition}


\section{A unitary transform from   
$L^2 ({\mathbb R}^{m}, d^mx) \otimes \C_{m} $ to a Hilbert
space of monogenic functions on $\R\times \R^{m}$}
\label{s-3}

We can view the CST 
unitary map $U$ in  the diagram (\ref{d2})
as a unitarization of the analytic continuation 
$\Cc$ from $\R^m$ to $\C^m$. Indeed, by precomposing
$\Cc$ with the smoothening contracting map
$e^{\frac{\Delta}2}$ one obtains a unitary 
isomorphism from a $L^2$ space on $\R^m$ to
a space of holomorphic square integrable functions on
$\C^m$, the complexification of $\R^m$. 

Aiming at obtaining an analogous 
unitarization 
of the Cauchy-Kowalewsky (CK) extension (\ref{ee-cke}), we precompose  
it with the same smoothening contracting map
$e^{\frac{\Delta}2}$ or, equivalently, we
substitute the vertical arrow
in the diagram (\ref{d2}) by the CK extension,

\begin{align}
 \label{d3}
\begin{gathered}
\xymatrix{
&&  \mathcal{M} (\R^{m+1}) \\
L^2(\R^m, d \un x)\otimes \C_{m}  \ar@{^{(}->}[rr]_{e^{\frac{\Delta}2}}  \ar@{^{(}->}[rru]^{V} && \widetilde \Aa (\R^m)\otimes \C_{m},
\ar[u]_{\, e^{-x_0 \un D}}
  }
\end{gathered}
\end{align}
where $\mathcal{M} (\R^{m+1})$ denotes the space of $\C_{m}$-valued monogenic functions on $\R^{m+1}$.
Let $\rho_1$ be the heat kernel in $\R^m$, as in Section \ref{cst},
\be
\rho_1(x) = (2\pi)^{-m/2}e^{-\frac{x^2}{2}} = (2\pi)^{-m} \int_{\R^m} e^{-\frac{p^2}{2}} 
e^{i(\un p, \un x)} d\un p.
\ee

From Proposition \ref{pu2}, we then have the Cauchy-Kowalevski extension of $\rho_1$,
$$
e^{-x_0\un D} \rho_1 (\un x) = (2\pi)^{-m} \int_{\R^m} e^{-\frac{p^2}{2}} 
e^{i(\un p, \un x)} e^{-ix_0 \un p}d\un p.
$$

Our main result in this Section is then
\begin{theorem}\label{bigT}For $\varphi\in L^2(\R^m, d\un x)\otimes \C_m$ define
\be\label{defV}
V(\varphi) (x_0,\un x) = (2\pi)^{-m}\int_{\R^m} \left(\int_{\R^m} e^{-\frac{p^2}{2}} 
e^{i(\un p, \un x-\un y)} e^{-ix_0 \un p}d\un p\right) \varphi(y)d\un y,
\ee
which can be abbreviated by
$$
V(\varphi) = e^{-x_0\un D} \circ e^{\frac{\Delta}{2}}\varphi.
$$
Then, in the diagram 
\begin{align}
 \label{d333}
\begin{gathered}
\xymatrix{
&&  \mathcal{M}L^2 (\R^{m+1}, d\mu)  \\
L^2(\R^m, d \un x)\otimes \C_{m}  \ar@{^{(}->}[rr]_{e^{\frac{\Delta}2}}  \ar[rru]^{V} && \widetilde \Aa (\R^m)\otimes 
\C_{m},
\ar[u]_{\, e^{-x_0 \un D}}
  }
\end{gathered}
\end{align}
the map $V:L^2(\R^m,d\un x)\otimes \C_m \to {\mathcal M}L^2(\R^{m+1}, d\mu)$ is  a unitary isomorphism 
of Hilbert spaces, where ${\mathcal M}L^2(\R^{m+1}, d\mu)$ is the 
Hilbert space of monogenic functions on $\R^{m+1}$ which are $L^2$ with respect to the measure 
$$
d\mu=\frac{1}{\sqrt{\pi}}e^{-x_0^2}dx_0d\un x,
$$
and where the standard inner product on $\C_m$ is considered.
\end{theorem}

Let us now list some direct consequences of this theorem.

\begin{corollary}The subspace of monogenic functions which are in $L^2(\R^{m+1},d\mu)\otimes \C_m$, 
${\mathcal M}L^2(\R^{m+1},d\mu)$,
is a closed subspace of $L^2(\R^{m+1},d\mu)\otimes \C_m$.
\end{corollary}

We also obtain a characterization of the range of the heat operator in terms of monogenic functions, 
which is analogous to the one given in terms of holomorphic functions by the Segal-Bargmann theorem 
(see, for example, \cite{Ha3}).

\begin{corollary}
A real analytic function $F$ on $\R^m$ is of the form 
$$
F=e^{\frac{\Delta}{2}}f,
$$
with $f\in L^2(\R^m, dx)$, iff its monogenic extension to $\R^{m+1}$ exists and is $d\mu$-square-integrable, 
$$
e^{-x_0\un D}f \in \mathcal{M}L^2 (\R^{m+1}, d\mu).
$$
\end{corollary}

Let now $\{H_k, k\in \N_0^m\}$ denote the orthogonal basis of $L^2(\R^m, e^{-x^2}dx)$ 
consisting of Hermite polynomials on $\R^m$, with
$$
||H_k||^2= \pi^{m/2} 2^k k!,
$$
where we use multi-index notation: $k\in \N^m_0$, 
\begin{equation}\label{hermites}
H_k(x):=H_{k_1}(x_1)\cdots H_{k_m}(x_m),
\end{equation} 
$2^k:=2^{k_1+\cdots+k_m}$,  
$k!:=k_1!\cdots k_m!$ and $x^k= x_1^{k_1}\cdots x_m^{k_m}$.

Defining, 
\begin{equation}\label{euro}
\varphi_k (x) = H_k(x) e^{-\frac{x^2}{2}},
\end{equation}
we have that $\{\varphi_k, k\in \N_0^m\}$ is an orthogonal basis for $L^2(\R^m, d\un x)$.
From the isometricity of $V$ and Lemma \ref{l2} below, we also obtain the following

\begin{corollary}\label{basis}Let
$$
\psi_k =2^{m/2}V(\varphi_k) =e^{-x_0\un D}\left(x^ke^{-\frac{x^2}{4}}\right), \,\,\, k\in \N_0^m.
$$
Then, the set
$$
\{\psi_k e_A, k\in \N_0^m, A\subset\{1,\dots,m\}\}
$$
is an orthogonal basis for the Hilbert space $\mathcal{M}L^2 (\R^{m+1}, d\mu)$, where $e_{\emptyset}=1$.
\end{corollary}

\begin{remark}
The coherent state transform of Hall is onto ${\mathcal H}L^2(\C^m,\nu dxdy)$. Let us consider its inverse
$$
U^{-1}: {\mathcal H}L^2(\C^m,\nu dxdy) \to L^2(\R^m,dx).
$$ 

By composing this operator with the operator $V$ above we obtain the operator
$$
V\circ U^{-1}:  {\mathcal H}L^2(\C^m,\nu dxdy) \to \mathcal{M}L^2 (\R^{m+1}, d\mu)
$$
which is a unitary isomorphism.
\end{remark}

We will prove theorem \ref{bigT} by a sequence of lemmas.

\begin{lemma}\label{l1}
If $\varphi \in L^2(\R^m, d\un x)\otimes \C_m$ then $V(\varphi)$ is a monogenic function on $\R^{m+1}$.
\end{lemma}

\begin{proof}
By Leibniz rule, both 
$$
\frac{\partial}{\partial x_0}V(\varphi),\,\,\,{\rm and}\,\,\, \un D V(\varphi)
$$
can be computed by taking the differential operators inside both integral symbols in (\ref{defV}), due to the presence of 
the gaussian factor in the integrand which ensures the integrability of all of its derivatives. This then 
implies that 
$$
\left(\frac{\partial}{\partial x_0}+\un D\right)V(\varphi) =0,
$$
so that $V(\varphi)$ is a monogenic function on $\R^{m+1}$. 
\end{proof}

\begin{lemma}\label{l2}We have
$$
e^{\frac{\Delta}{2}}\varphi_k = 2^{-m/2}x^k e^{-\frac{x^2}{4}}.
$$
\end{lemma}
\begin{proof}
This follows from the well-known identity
$$
e^{2xy-y^2}=\sum_{l\in \N_0^m} H_l(x)\frac{y^l}{l!}, \,\,\,\,\, x,y\in \R^m
$$
and from
$$
\langle H_k, H_l\rangle_{L^2(\R^m, e^{-x^2}dx)} = \pi^{m/2} 2^k k!\delta_{kl}.
$$
A simple evaluation of gaussian integrals then gives the result
$$
e^{\frac{\Delta}{2}}\varphi_k = \int_{\R^m} e^{-\frac{(x-y)^2}{2}} H_k(y) e^{-\frac{y^2}{2}}dy = 2^{-m/2}x^k e^{-\frac{x^2}{4}}.
$$
\end{proof}

\begin{lemma}
\label{ee-scst}
Let $f = \sum_A \, f_A \, e_A \, \in \Ss(\R^m, d\un x) \otimes \C_{m}$, with Fourier transform $\wh f$.
Then,
\ba
\label{ecthar}
\nonumber && V(f)(x_0, \un x)  = 
   \frac 1{(2\pi)^{m/2}} \, \int_{\R^m}  \, e^{- \frac {|\un p|^2}2} \, e^{i ((\un p, \un x) - x_0 \un p)} \,  \wh f(\un p) \, d\un p .
\ea
\end{lemma}
\begin{proof}
For $f\in \Ss(\R^m)\otimes \C_{m},$ we have
$$
f(\un x)= \frac 1{(2\pi)^{m/2}} \, \int_{\R^m} \, e^{i (\un p, \un x)} \, \wh f(\un p) \, d\un p  \,
=  \frac 1{(2\pi)^{m/2}} \, \int_{\R^m} \, e^{i (\un p, \un x)} \,  \sum_A \wh f_A(\un p) \,   e_A \, d\un p 
\, \,   \, . 
$$
and the result follows from  (\ref{ee-cke}),
$$
e^{\frac{\Delta}2} e^{i(\un x,\un p)} = e^{-\frac{|p|^2}{2}}e^{i(\un x,\un p)},
$$
and the fact that under the conditions of the proposition the heat operator can be taken inside the integral.
\end{proof}

\begin{remark} We also note the following useful expression,
\ba   
 \nonumber &&   V(f)(x_0,\un x)=\frac 1{(2\pi)^{m/2}} \int_{\R^m}  \, e^{- \frac {|\un p|^2}2} \, e^{i (\un p, \un x) } \,\left(
\cosh(x_0 |\un p|) -i  \sinh(x_0 |\un p|) \, \frac{\un p}{|\un p|} \right)\,  \wh f(\un p) \, d\un p.  
\ea
\end{remark}

\begin{lemma}
\label{l4}
For $f,h \in \Ss(\R^m) \otimes \C_{m}$, with $f=\sum_A f_A e_A, h=\sum_Ah_A e_A$, we have
$$
\langle V(f), V(g)\rangle_{L^2(\R^{m+1}, d\mu)\otimes \C_m} = \langle f,g\rangle_{L^2(\R^m, d\un x)\otimes \C_m}.
$$
\end{lemma}
\begin{proof}
For any $1$--vector $\un p = \sum_{j=1}^m \, p_j e_j  \in  \R_{m}^1$
one has
\be
\label{ee-ip1v} 
(\un p \, u, v) =  -( u, \un p \, v)  \, , \quad    \forall u, v \in \C_{m}
\ee
and therefore
$$
(e^{i  \un p} \, u, v) = ( u, e^{i  \un p} \, v)  \, ,  \quad    \forall u, v \in \C_{m}, 
$$
where the hermiticity of the standard inner product in $\C_{m}$, $( \cdot , \cdot )$, is used.

Then, for $f,h \in \Ss(\R^m) \otimes \C_{m}$, with $f=\sum_A f_A e_A, h=\sum_Ah_A e_A$ we have, from 
Lemma \ref{ee-scst},
\bas
&& \langle V(f), V(h)\rangle_{L^2(\R^{m+1}, d\mu)\otimes \C_m} = \\
&=&
\frac 1{\sqrt{\pi}(2\pi)^{m}}
  \int_{\R^{m+1}  \times \R^{2m}} \, 
e^{i (\un p - \un q, \un x)} \, e^{- \frac{|p|^2+|q|^2}2} \,  \left(e^{-ix_0 \un p} \, \wh f(\un p),
e^{-ix_0 \un q} \, \wh h(\un q)\right)
\,  e^{-x_0^2} \, 
dx_0 d\un x   d\un p d\un q = \\
&=&
\frac 1{\sqrt{\pi}}
  \int_{\R \times \R^m} \, 
 \, e^{- {|p|^2}} \,  \left(e^{-ix_0 \un p} \, \wh f(\un p),
e^{-ix_0 \un p} \, \wh h(\un p)\right)
\,  e^{-x_0^2} \, 
dx_0  d\un p = \\
&=&
\frac 1{\sqrt{\pi}}
  \int_{\R \times \R^m} \, 
 \, e^{- {|p|^2}} \,  \left(e^{-2ix_0 \un p} \, \wh f(\un p),
 \, \wh h(\un p)\right)
\,  e^{-x_0^2} \, 
dx_0  d\un p = \\
&=&
\frac 1{\sqrt{\pi}}
  \int_{\R \times \R^m} \, 
 \, e^{- {|p|^2}} \,  \left[\cosh(2x_0|\un p|) \left(\wh f(\un p), \, \wh h(\un p)\right)
-i \frac{\sinh(2x_0|\un p|)}{|\un p|}  \, \left(\un p \, \wh f(\un p), \, \wh h(\un p)\right)\right]
\,  e^{-x_0^2} \, 
dx_0  d\un p = \\
&=&
\frac 1{\sqrt{\pi}}
  \int_{\R \times \R^m} \, 
 \, e^{- {|p|^2}} \,  \cosh(2x_0|\un p|) \left(\wh f(\un p), \, \wh h(\un p)\right)
\,  e^{-x_0^2} \, 
dx_0  d\un p = \\
&=&
  \int_{ \R^m} \, 
 \, \left(\wh f(\un p), \, \wh h(\un p)\right) d\un p = \\
&=& \langle f, h\rangle_{L^2(\R^m, d\un x)\otimes \C_m}.
\eas
\end{proof}

\begin{proof}{\it (of Theorem \ref{bigT})}
{}From the denseness of $\Ss(\R^m)$ in $L^2(\R^m,d \un x)$ we conclude from Lemma \ref{l4} that $V$ is an 
isometry onto its image which is, therefore, closed in $L^2(\R^{m+1}, d\mu)\otimes \C_m$. 
Moreover, Lemma
\ref{l1} ensures that the image of $V$ contains only functions which are monogenic on $\R^{m+1}$. Therefore, 
$V(L^2(\R^m,d \un x)\otimes \C_m)$ is a Hilbert space of monogenic functions.

To prove that the image of $V$ is 
all of $\mathcal{M}L^2 (\R^{m+1}, d\mu)$, note that the restriction of  
$f\in \mathcal{M}L^2 (\R^{m+1}, d\mu)$ to the 
hyperplane $x_0=0$, $f_0(\un x)= f(x_0=0, \un x)$, determines $f$ uniquely. Since entire monogenic functions 
have a Taylor series with infinite radius of convergence (see, for example, \cite{BDS, Som}), it follows that 
$f_0$ can be expressed uniquely in the form
$$
f_0 = \sum_A\sum_{k\in \N_0^m} \alpha_{k,A} x^k e^{-\frac{x^2}{4}}e_A, \,\,\,\,\,\, \alpha_{k,A}\in \C. 
$$
Now, from Proposition \ref{initial},
$$
f(x_0,\un x) = \sum_A\sum_{j=0}^\infty\frac{(-x_0)^j}{j!} \un D^j 
\left(\sum_{k\in \N_0^m} \alpha_{k,A} x^k e^{-\frac{x^2}{4}}\right)\,e_A=
\sum_A\sum_{j=0}^\infty\sum_{k\in \N_0^m}\frac{(-x_0)^j}{j!} 
\alpha_{k,A}\un D^j\left(x^k e^{-\frac{x^2}{4}}\right)\,e_A,
$$
since convergent power series can be differentiated term by term. (The Gaussian factor $e^{-\frac{x^2}{4}}$ could, of course, also be expanded in power series.) This series converges absolutely in 
all of $\R^{m+1}$ so that the two summations can be interchanged giving
$$
f(x_0,\un x)=\sum_A\sum_{k\in \N_0^m} \alpha_{k,A} e^{-x_0\un D}\left(x^ke^{-\frac{x^2}{4}}\right)\,e_A.
$$
From Lemma \ref{l2}, all partial sums
$$
\sum_A\sum_{k\in \N_0^m: \vert\vert k\vert\vert<N} \alpha_{k,A} e^{-x_0\un D}\left(x^ke^{-\frac{x^2}{4}}\right)\,e_A
$$
are in the image of $V$ which is closed.  Moreover, 
see also Corollary \ref{basis} above, from Lemma \ref{l4} and from the orthogonality of the set 
$\{\varphi_ke_A, k\in \N_0^m, A\subset\{1,\dots,m\}\}$ in $L^2(\R^m, d\un x)\otimes \C_m$, we obtain that 
the $L^2$ condition for $f$ 
in $\R^{m+1}$ with respect to the measure $d\mu$ is equivalent to 
$$
\sum_A\sum_{k\in \N_0^m} \alpha_{k,A} \varphi_k e_A\in L^2(\R^m, d\un x)\otimes \C_m.
$$
This completes the proof of Theorem \ref{bigT}.
\end{proof}

We finish this section with a Proposition on explicit expressions for $V(P_ke^{-\frac{x^2}{2}})$ where 
$P_k$ is an homogeneous monogenic polynomial of degree $k$ in $\R^m$.

\begin{proposition}Let $P_k$ be an homogeneous monogenic polynomial of degree $k$ in $\R^m$.
Then
$$
V\left(P_k e^{-\frac{x^2}{2}}\right) = e^{-\frac{x^2}{2}}\sum_{n=0}^\infty \frac{1}{n!}\left(\sum_{j=0}^{[n/2]} 
\frac{(-1)^j}{2^jj!} u(n,j)!\, x_0^{u(n,j)}\right)
H_{n,m,k}(\un x) P_k(\un x),
$$
where $u(n,j)=(n-2j)$ for $n$ even and $u(n,j)=(n-2j-1)$ for $n$ odd and the $H_{n,m,k}$ are the so-called generalized Hermite polynomials. 
\end{proposition}

\begin{proof}
From Chapter III of \cite{DSS}, we have
$$
e^{-x_0\un D} \left(P_k e^{-\frac{x^2}{2}}\right)(x_0,\un x) 
= e^{-\frac{x^2}{2}}\sum_{n=0}^\infty \frac{x_0^n}{n!} H_{n,m,k}(\un x) P_k(\un x).
$$
Since the operator 
$\un D$ commutes with the Laplace operator on $\R^m$, $\Delta$, we have
$$
V = e^{-x_0\un D} \circ e^{\frac{\Delta}{2}} = e^{\frac{\Delta}{2}} \circ  e^{-x_0\un D}.
$$
On the other hand, since monogenic functions are harmonic we have
$$
\Delta f = -\frac{\partial^2f}{\partial x_0^2},
$$
for $f$ monogenic on $\R^{m+1}$. Therefore,
$$
V\left(P_k e^{-\frac{x^2}{2}}\right) = e^{-\frac12 \frac{\partial^2}{\partial x_0^2}} 
\left(e^{-\frac{x^2}{2}}\sum_{n=0}^\infty \frac{x_0^n}{n!} H_{n,m,k}(\un x) P_k(\un x)\right)=
\left(e^{-\frac{x^2}{2}}\sum_{n=0}^\infty \left(e^{-\frac12 \frac{\partial^2}{\partial x_0^2}} 
 \frac{x_0^n}{n!}\right) H_{n,m,k}(\un x) P_k(\un x)\right)
$$
which proves the proposition.
\end{proof}

\section{{A unitary transform from  $L^2({\T}^{m}, d^mx) \otimes \C_{m}$ to a Hilbert space of monogenic functions on $\R\times \T^{m}$}}
\label{s-4}

In this section, we generalize the coherent state transform of the last section to a transform on $\C_{m}$-valued 
$L^2$ functions on the compact Lie group defined by the $m$-dimensional torus
$$
\T^m = \R^m/\Z^m,
$$
where $x\sim  x+2\pi k, k\in \Z^m,  x\in \R^m.$ We will still denote by $(x_1,\dots,x_m)$ the periodic 
coordinates on $\T^m$, 
with $x_j\in [0,2\pi], j=1,\dots, m$. Note that the definitions of the Dirac ($\tilde D$), Cauchy-Riemann ($D$) and $\un D$ operators
generalize straightforwardly. Likewise, the Cauchy-Kowaleski extension of section \ref{ss-ca} is obtained from the same 
expression. Let $\Delta$ denote the Laplacian on $\T^m$ with respect to an invariant metric and 
let $dx$ denote the 
unit volume Haar measure.

We then define the operator $V$ by the following diagram

\begin{align}
 \label{d31}
\begin{gathered}
\xymatrix{
&&  \mathcal{M} (\T^{m}\times \R) \\
L^2(\T^m, d x)\otimes \C_{m}  \ar@{^{(}->}[rr]_{e^{\frac{\Delta}2}}  \ar@{^{(}->}[rru]^{V} && \widetilde \Aa (\T^m)\otimes \C_{m},
\ar[u]_{\, e^{-x_0 \un D}}
  }
\end{gathered}
\end{align}

where $\Aa (\T^m)$ denotes the space of real analytic functions on $\T^m$ and $\mathcal{M} (\T^{m}\times \R)$ 
is the space of $\C_{m}$-valued monogenic functions on $\T^m\times \R$. 

As in the previous section, we obtain an explicit expression for this operator using the Fourier, or Peter-Weyl, 
decomposition for functions in $L^2(\T^m,d x)$. For $k\in \Z^m$, define $\un k = \sum_{j=1}^m k_je_j\in \R_m^1.$

\begin{proposition}
\label{ee-scst2}
Let $f = \sum_A \, f_A \, e_A \, \in L^2(\T^m, \un dx) \otimes \C_{m}$, with Fourier decomposition
$$
f(\un x)= \frac 1{(2\pi)^{m/2}} \, \sum_{k\in \Z^m}  \, f_k e^{i (\un k, \un x)} \,  \,
=  \frac 1{(2\pi)^{m/2}} \, \sum_{k\in \Z^m} \, e^{i (\un k, \un x)} \,  \sum_A  f_{k,A} \,   e_A \,  
\, \,   \, . 
$$
Then,
\ba
\label{ecthar2}
\nonumber &&  
 V(f)(x_0,\un x)=  \frac 1{(2\pi)^{m/2}} \, \sum_{k\in \Z^m}  \, e^{- \frac {|k|^2}2} \, e^{i ((\un k, \un x) - x_0 \un k)} \,   f_m.
\ea
\end{proposition}

\begin{proof}
The proof is analogous to the one of Lemma \ref{l4}.
\end{proof}

\begin{remark}We note the following useful formula,
\ba    
 \nonumber && V(f)(x_0, \un x)  =   \frac 1{(2\pi)^{m/2}}  \, \sum_{k\in \Z^m} \, e^{- \frac {|\un k|^2}2} \, e^{i (\un k, \un x) } \,\left(
\cosh(x_0 |\un k|)-i\sinh(x_0 |\un k|) \, \frac{\un k}{|\un k|}\right)  \,  f_k.   
\ea
\end{remark}

Consider now the measure on $\R^{m}\times \R$  given by
$$
d\mu   =  {\frac{1}{\sqrt{\pi} }} \,    \, {e^{-  x_0^2}} \, dx_0 dx ,
$$
and let $\mathcal{M}L^2 (\T^{m}\times \R, d\mu)$ be the corresponding Hilbert space of $L^2$ monogenic functions.

The analog of Theorem \ref{bigT} is now
\begin{theorem}
\label{th-scst2}
The map $V$ in diagram (\ref{d31}) is unitary with respect to the measure $d\mu$, 
 i.e. the map
$V$ in the diagram
\begin{align}
 \label{d3333}
\begin{gathered}
\xymatrix{
&&  \mathcal{M}L^2 (\T^{m}\times \R, d\mu)  \\
L^2(\T^m, d x)\otimes \C_{m}  \ar@{^{(}->}[rr]_{e^{\frac{\Delta}2}}  \ar[rru]^{V} && \widetilde \Aa (\T^m)\otimes 
\C_{m}
\ar[u]_{\, e^{-x_0 \un D}}
  }
\end{gathered}
\end{align}
is a unitary isomorphism.
\end{theorem}

\begin{proof}
Then, for $f,h \in L^2(\T^m) \otimes \C_{m}$, with $f=\sum_A f_A e_A, h=\sum_Ah_A e_A$ we have
\bas
&& \langle V(f), V(h)\rangle_{L^2(\T^m) \otimes \C_{m}} = \\
&=&
\frac 1{\sqrt{\pi}(2\pi)^{m}}
  \int_{\T^{m}\times \R}  \sum_{k,k'\in \Z^m} \, 
e^{i (\un k - \un k', \un x)} \, e^{- \frac{|k|^2+|k'|^2}2} \,  \left(e^{-ix_0 \un k} \, f_k,
e^{-ix_0 \un k'} \, h_{k'}\right)
\,  e^{-x_0^2} \, 
dx_0 d\un x  = \\
&=&
\frac 1{\sqrt{\pi}} \int_{\R}\sum_{k\in \Z^m}
  \, 
 \, e^{- {|k|^2}} \,  \left(e^{-ix_0 \un k} \, f_k,
e^{-ix_0 \un k} \, h_k\right)
\,  e^{-x_0^2} \, 
dx_0  = \\
&=&
\frac 1{\sqrt{\pi}}\int_{\R}\sum_{k\in \Z^m}
   \, 
 \, e^{- {|k|^2}} \,  \left(e^{-2ix_0 \un k} \, f_k,
 \, h_k\right)
\,  e^{-x_0^2} \, 
dx_0  = \\
&=&
\frac 1{\sqrt{\pi}} \int_{\R}\sum_{k\in \Z^m}
   \, 
 \, e^{- {|k|^2}} \,  \left[\cosh(2x_0|\un k|) \left(f_k, \, h_k\right)
-i \frac{\sinh(2x_0|\un k|)}{|\un k|}  \, \left(\un k \, f_k, \, h_k\right)\right]
\,  e^{-x_0^2} \, 
dx_0  = \\
&=&
\frac 1{\sqrt{\pi}}\int_{\R}\sum_{k\in \Z^m}
   \, 
 \, e^{- {|k|^2}} \,  \cosh(2x_0|\un k|) \left(f_k, \, h_k\right)
\,  e^{-x_0^2} \, 
dx_0  = \\
&=&
  \sum_{k\in \Z^m} \, 
 \, \left(f_k, \, h_k\right)  = \\
&=& \langle f, h\rangle_{L^2(\T^m,dx)\otimes\C_m}.
\eas
The proof of ontoness is analogous to the one in the proof of Theorem \ref{bigT}.
\end{proof}




\section{Quantum mechanical interpretation}
\label{s-6}

As is well known, the Schr\"odinger representation in quantum mechanics is the one for which the position 
operator $\hat x$ acts by multiplication on $L^2(\R^m,dx)$. The momentum operator is then given by 
$$
\hat p_j = i \frac{\partial}{\partial x_j},\,\, j=1,\dots, m.
$$

The CST from Section \ref{cst} intertwines the Schr\"odinger 
representation with the Segal-Bargmann representation, on which
the operators $\hat x_j+i\hat p_j$ acts as the operator of multiplication 
by the  holomorphic function $x_j+ip_j$ (see Theorem 6.3 of \cite{Ha2}),

\be\label{cstacts}
\left( U\circ (\hat x_j +i\hat p_j) \circ U^{-1}\right)(f)(x,p)=(x_j+ip_j)f(x,p), \, j=1,\dots, m.
\ee

We will prove now the analogous result for the coherent state transform of section \ref{s-3}.



\begin{theorem}
\label{th-scts} 
The unitary map $V$ induces a representation of the observable $\un x+i\un p$ on the Hilbert space 
of monogenic functions ${\mathcal M}L^2(\R^{m+1},d\mu)$ given by
\be\label{scstacts}
\left (V \circ ({\hat{\un x}} + i {\hat{\un p}}) \circ V ^{-1}\right)(f)=
\left(e^{-x_0 \un D} \circ \un x \right) \,(f_0),
\ee
where on the right hand side the operator $\un x$ acts by Clifford multiplication on the left and
$f_0(\un x) = f(x_0=0, \un x)$. 
\end{theorem}

\begin{proof}
Let $h = \sum_A \, h_A \, e_A \, \in \Ss(\R^m)\otimes \C_m$ and let $f=V(h)$. 
Then,
$$
({\hat{\un x}} + i {\hat{\un p}})(h)(\un x) =  \left(\frac 1{(2\pi)^{m/2}} \int_{\R^m}  (\un x +i\un p)\, e^{i (\un p, \un x) } \,  \wh h(\un p) \, d\un p  \right).
$$
From the above result of Hall (see also \cite{Ha2}),
$$
e^{\frac{\Delta}{2}}((\un x + i\un p) e^{i (\un p, \un x) }) = \un x e^{\frac{\Delta}{2}}e^{i (\un p, \un x) },
$$
from which the result follows from the denseness of $\Ss(\R^m)\otimes \C_{m}$ in $L^2(\R^m,d\un x)\otimes \C_{m}$.
\end{proof}

\begin{remark}Notice that the Segal-Bargmann transform can be expressed as, from (\ref{ee-aco}),
$$
e^{i\sum_{k=1}^m y_k\frac{\partial}{\partial x_k}}\circ e^{\frac{\Delta}{2}}
$$
while the transform $V$ of Section \ref{s-3} is given by
$$
e^{-\sum_{k=1}^m x_0\,e_k\frac{\partial}{\partial x_k}}\circ e^{\frac{\Delta}{2}}.
$$ 
We therefore see that $V$ is obtained from the Segal-Bargmann transform by replacing, in the operator of analytic continuation (\ref{ee-aco}), the momentum variables $y_k$ by the non-commutative variables $-ix_0 e_k$, where $x_0$ 
is the euclidean time.
\end{remark}

\bigskip
\bigskip

{\bf \large{Acknowledgements:}} 
JM and JPN thank Pedro Gir\~ao and Jorge Silva for helpful
discussions. 

The authors were partially
supported by Macau Government FDCT through the project 099/2014/A2,
{\it Two related topics in Clifford analysis},
and by the University of Macau Research Grant  MYRG115(Y1-L4)-FST13-QT. 
 JM and JPN were also partly supported by
FCT/Portugal through the projects UID/MAT/04459/2013, 
EXCL/MAT-GEO/0222/2012, PTDC/MAT-GEO/3319/2014 and by the (European Cooperation in Science and 
Technology) COST Action MP1405 QSPACE.


\begin{thebibliography}{FMMMM}

\bibitem[Ba]{Ba} V. Bargmann, \emph{On a Hilbert space of analytic functions
and an associated integral transform, Part I}, Comm. Pure Appl. Math. {\bf 14} (1961),
187--214.

\bibitem[BDS]{BDS}
F. Brackx, R. Delanghe and F. Sommen, \emph{Clifford analysis}, Research Notes
in Mathematics, {\bf 76}, Pitman, Boston, 1982.

\bibitem[CLSS]{CLSS}
F. Colombo, R. Lavicka, I. Sabadini and V. Soucek,
\emph{The Radon transform between monogenic and generalized slice monogenic functions}, Math. Ann. 
DOI 10.1007/s00208-015-1182-3.

\bibitem[CSS1]{CSS1}
F. Colombo, I. Sabadini and D.C. Struppa, \emph{Slice monogenic functions,} Israel J. Math. {\bf 171} (2009), 385--403.

\bibitem[CSS2]{CSS2}
F. Colombo, I. Sabadini and D.C. Struppa, \emph{An extension theorem for slice monogenic functions and some of its consequences}, Israel J. Math. {\bf 177} (2010), 369--389.

\bibitem[CSS3]{CSS3}
F. Colombo, I. Sabadini and D.C. Struppa, \emph{Noncommutative Functional
Calculus}, Birkh\"auser, 2011.

\bibitem[DG]{DG} K.~Diki, A.~Ghanmi, \emph{A quaternionic analogue of the Segal-Bargmann transform}, 
arXiv:1603.05052. 


\bibitem[DS]{DS} N. De Schepper and F. Sommen,   \emph{Cauchy--Kowalevski extensions and monogenic
plane waves using spherical monogenics,} Bull. Braz. Math. Soc. {\bf 44} (2013),  321--350.

\bibitem[DSS]{DSS}
 R. Delanghe, F. Sommen and V. Soucek, \emph{Clifford algebra and spinor--valued functions},
Mathematics and its Applications, {\bf 53},  Kluwer, 1992.

\bibitem[Dr]{Dr} B.~Driver, \emph{On the Kakutani-It\^o-Segal-Gross and Segal-Bargmann-Hall isomorphisms},
J. Funct. Anal. \textbf{133} (1995), 69--128.

\bibitem[F]{F} R.~Fueter, \emph{Die Funktionentheorie der Differetialgleichungen $\Delta u = 0$ und $\Delta \Delta u 
= 0$ mit vier reellen Variablen}, Comm. Math. Helv. \textbf{7} (1935), 307--330.

\bibitem[Ha1]{Ha1}
B.~C. Hall, \emph{The {S}egal-{B}argmann ``coherent-state'' transform for Lie
  groups}, J. Funct. Anal. \textbf{122} (1994), 103--151.

\bibitem[Ha2]{Ha2}
B.~C. Hall, \emph{Holomorphic methods in analysis and mathematical physics},  First Summer School
in Analysis and Mathematical Physics (Cuernavaca Morelos, 1998). Contemp. Math., 260:1--59,
2000.

\bibitem[Ha3]{Ha3}B.C.~Hall, \emph{The range of the heat operator}, in ``The ubiquitous heat kernel'', Jorgenson, Jay et al. Eds., AMS special session, Boulder, CO, USA, October 2--4, 2003. 
Providence, RI: American Mathematical Society (AMS) (ISBN 0-8218-3698-6/pbk). Contemporary Mathematics 
398, 203-231 (2006).


\bibitem[KMNQ]{KMNQ} W.~D.~Kirwin, J.~Mour\~ao, J.~P. Nunes and T.~Qian, \emph{Extending coherent state 
transforms to Clifford analysis}, arXiv:1601.01380.


\bibitem[KQS]{KQS} K.~I. Kou, T.  Qian and F.  Sommen, 
\emph{Generalizations of Fueter's theorem},  Methods Appl. Anal.
\textbf{9} (2002),  273--289.

\bibitem[LMQ]{LMQ} C. Li, A. McIntosh and T. Qian, 
\emph{Clifford algebras, Fourier transforms and singular convolution operators on
Lipsschitz surfaces},  Rev. Mat. Iberoam.
\textbf{19} (1994),  665--721.


\bibitem[PQS]{PQS} D.~Pe\~na Pe\~na, T.~Qian and F.~Sommen, \emph{An alternative proof of Fueter's theorem},
Complex Var. Elliptic Equ. \textbf{51} No. 8-11 (2006) 913--922.

\bibitem[Q1]{Q1} T.~Qian, \emph{Generalization of Fueter's result to  $\R^{n+1}$.}
Atti Accad. Naz. Lincei Cl. Sci. Fis. Mat. Natur. Rend. Lincei (9) Mat. Appl. \textbf{8} (1997), 111--117.

\bibitem[Q2]{Q2} T.~Qian, \emph{Fourier analysis on starlike Lipschitz surfaces.}
J. Fun. Anal.
 \textbf{183} (2001), 370--412.



\bibitem[Se1]{Se1} I.~Segal, \emph{ Mathematical characterization of the physical vacuum for a linear Bose-Einstein
field}, Illinois J. Math. {\bf 6} (1962), 500--523.

\bibitem[Se2]{Se2} I.~Segal, \emph{The complex wave representation of the free Boson field}, in 
``Topics in functional
analysis: Essays dedicated to M.G. Krein on the occasion of his 70th birthday'', I. Gohberg
and M. Kac, Eds, Advances in Mathematics Supplementary Studies, Vol. 3, pp. 321--343.
Academic Press, New York, 1978.

\bibitem[Som]{Som} F.~Sommen, \emph{Some connections between Clifford analysis and complex analysis}, Complex Variables 
\textbf{1} (1982), 97--118.

\bibitem[Sou]{Sou} V. Sou\v{c}ek, \emph{Generalized C-R equations on manifolds}, in ``Clifford algebras and their 
applications in mathematical physics'', ed. J.S.R.Chisholm and A.K.Common, NATO ASI Series C Vol. 183, D.Reidel Pub. Company, 1986.  

\end{thebibliography}
\end{document}